\documentclass[12pt]{amsart}

\title{The cycle structure of unicritical polynomials}
\author{Andrew Bridy}
\address{Department of Mathematics, Texas A\&M University}
\email{\href{mailto:andrewbridy@math.tamu.edu}{andrewbridy@math.tamu.edu}}
\author{Derek Garton}
\address{Fariborz Maseeh Department of Mathematics and Statistics, Portland State University}
\email{\href{mailto:gartondw@pdx.edu}{gartondw@pdx.edu}}
\date{\today}



\usepackage{amssymb,url,MnSymbol,bbm}

\usepackage{hyperref}
\hypersetup{breaklinks=true,
colorlinks=true,
urlcolor=blue,
linkcolor=cyan
}

\usepackage[top=1in,bottom=1in,left=1in,right=1in]{geometry}

\usepackage[capitalise,nameinlink]{cleveref}

\makeatletter
\def\imod#1{\allowbreak\mkern8mu{\operator@font mod}\,\,#1}
\makeatother

\newcommand{\Z}{\ensuremath{{\mathbb{Z}}}}

\newcommand{\Q}{\ensuremath{{\mathbb{Q}}}}
\newcommand{\R}{\ensuremath{{\mathbb{R}}}}
\newcommand{\F}{\ensuremath{{\mathbb{F}}}}

\newcommand{\lv}{\ensuremath{\left\vert}}
\newcommand{\rv}{\ensuremath{\right\vert}}
\newcommand{\lp}{\ensuremath{\left(}}
\newcommand{\rp}{\ensuremath{\right)}}
\newcommand{\lb}{\ensuremath{\left\{}}
\newcommand{\rb}{\ensuremath{\right\}}}
\newcommand{\lc}{\ensuremath{\left[}}
\newcommand{\rc}{\ensuremath{\right]}}
\newcommand{\la}{\ensuremath{\left\langle}}
\newcommand{\ra}{\ensuremath{\right\rangle}}

\DeclareMathOperator{\Gal}{Gal}

\DeclareMathOperator{\Fix}{Fix}

\DeclareMathOperator{\M}{\mathbb{M}}
\DeclareMathOperator{\id}{id}

\theoremstyle{plain}
\newtheorem{finalmomentscor}{Corollary}[section]

\newtheorem{fixedpointprob}[finalmomentscor]{Theorem}
\newtheorem{fixedpointprobcor}[finalmomentscor]{Corollary}

\newtheorem{recursion}[finalmomentscor]{Lemma}\newtheorem*{namedthm}{\namedthmname}
\newcounter{namedthm}
\newtheorem{rootsandcycles}[finalmomentscor]{Proposition}
\newtheorem{cycledistribution}[finalmomentscor]{Theorem}

\newtheorem{momenttheorem}[finalmomentscor]{Theorem}
\newtheorem{momentswecareabout}[finalmomentscor]{Corollary}
\newtheorem{stirlinglemma}[finalmomentscor]{Lemma}
\newtheorem{bigtheorem}[finalmomentscor]{Theorem}

\newtheorem{thm}[finalmomentscor]{Theorem}
\newtheorem{disconnectedbox}[finalmomentscor]{Corollary}
\newtheorem{lotsofcycles}[finalmomentscor]{Corollary}
\newtheorem{whatisgamma}[finalmomentscor]{Conjecture}

\theoremstyle{remark}
\newtheorem{pmoments}[finalmomentscor]{Fact}
\newtheorem{rknfax}[finalmomentscor]{Remark}

\newtheorem{HITremark}[finalmomentscor]{Remark}
\newtheorem{FDTremark}[finalmomentscor]{Remark}
\newtheorem{momentremark}[finalmomentscor]{Remark}
\newtheorem{impossible}[finalmomentscor]{Remark}
\newtheorem{notsoimpossible}[finalmomentscor]{Remark}
\newtheorem{oldbox}[finalmomentscor]{Remark}

\theoremstyle{definition}
\newtheorem{rkn}[finalmomentscor]{Definition}

\newtheorem{komegan}[finalmomentscor]{Definition}

\makeatletter
\newenvironment{named}[1]
  {\def\namedthmname{#1}%
   \refstepcounter{namedthm}%
   \namedthm\def\@currentlabel{#1}}
  {\endnamedthm}
\makeatother

\begin{document}

\begin{abstract}
A polynomial with integer coefficients yields a family of dynamical systems indexed by primes as follows: for any prime $p$, reduce its coefficients mod $p$ and consider its action on the field $\F_p$.
The questions of whether and in what sense these families are random have been studied extensively, spurred in part by Pollard's famous ``rho'' algorithm for integer factorization (the heuristic justification of which is the randomness of one such family).
However, the cycle structure of these families \emph{cannot} be random, since in any such family, the number of cycles of a fixed length in any dynamical system in the family is bounded.
In this paper, we show that the cycle statistics of many of these families are \emph{as random as possible}.
As a corollary, we show that most members of these families have many cycles, addressing a conjecture of Mans et.\ al.
\end{abstract}

\maketitle

\section{Introduction}
\label{intro}

A \emph{\textup{(}discrete\textup{)} dynamical system} is a pair $\lp S,f\rp$ consisting of a set $S$ and a function $f\colon S\to S$.
For notational convenience, for any $n\in\Z_{>0}$, we let $f^n=\overbrace{f\circ\cdots\circ f}^{n\text{ times}}$; furthermore, we set $f^0=\id_S$.
We will denote the set of rational primes by $\mathcal{P}$.
For $f\in\Z[x]$ and $p\in\mathcal{P}$, write $\lc f\rc_p$ for the polynomial in $\F_p[x]$ obtained by reducing the coefficients of $f$ mod $p$.
Similarly, let $[f]_{\mathcal{P}}$ be the family of dynamical systems $\lb\lp\F_p,[f]_p\rp\mid p\in\mathcal{P}\rb$.
Spurred in part by Pollard's rho algorithm~\cite{Pollard}, the following question presents itself: for $f\in\Z[x]$, in what sense does the family $[f]_{\mathcal{P}}$ behave randomly?
Of course, the answer to this question depends upon
\begin{itemize}
\item
the choice of $f\in\Z[x]$ and
\item
what ``behave randomly'' means.
\end{itemize}
In this paper, we restrict our attention to
\begin{itemize}
\item
monic binomial unicritical polynomials (that is, polynomials of the form $x^k+a$ for $k\in\Z_{\geq2}$ and $a\in\Z$) and
\item
the cycle structure of dynamical systems.
\end{itemize}

Recall that if $\lp S,f\rp$ is a dynamical system, and $s\in S$ has the property that there is some $n\in\Z_{>0}$ with $f^n(s)=s$, we say that $s$ is \emph{periodic} or a \emph{periodic point} of $(S,f)$.
The smallest such $n$ is the \emph{period} of $s$, and if $s$ happens to be a point of period $n$, we call $\lb f^i(s)\mid i\in\Z_{\geq0}\rb$ an \emph{$n$-cycle}.
In~\cite{Gon} (see also the translation~\cite{GonTran}), Gon\v{c}arov discovered the cycle structure of random dynamical systems, which we now recall.
To ease notation, for any sets $S$ and $T$, let $S^T$ denote the set of functions from $T$ to $S$, and for any $X\in\Z_{\geq1}$, let $[X]=\lb1,\ldots,X\rb$.
Gon\v{c}arov proved that for any $n\in\Z_{\geq1}$, the function
\begin{align*}
\mu_n\colon\Z_{\geq0}&\to\R_{\geq0}\\
j&\mapsto\lim_{X\to\infty}
{\frac{\lv\lb f\in[X]^{[X]}\mid\lp[X],f\rp\text{ has precisely $j$ $n$-cycles}\rb\rv}
{\lv[X]^{[X]}\rv}}
\end{align*}
is the Poisson distribution with mean $\frac{1}{n}$.
As the Poisson distribution will play an important role in this paper, we pause here to introduce some notation: for any $\lambda\in\R_{>0}$, we let $\rho_\lambda\colon\Z_{\geq0}\to\R_{\geq0}$ be the Poisson distribution of mean $\lambda$.
With this notation in hand, we can rephrase Gon\v{c}arov's result: if $n\in\Z_{\geq1}$, then $\mu_n=\rho_{1/n}$.

Since Gon\v{c}arov was kind enough to compute the cycle statistics of random dynamical systems, we now turn to quantifing the cycle statistics of the families $[f]_{\mathcal{P}}$ for $f\in\Z[x]$; we do this by using the natural density on $\mathcal{P}$---which we will denote by $\delta$.
Specifically, for any subset $P\subseteq\mathcal{P}$, let
\[
\delta\lp P\rp:=\lim_{X\to\infty}
{\frac{\lv\lb p\in\mathcal{P}
\mid p\leq X\text{ and }p\in P\rb\rv}
{\lv\lb p\in\mathcal{P}\mid p\leq X\rb\rv}}
\hspace{20px}\text{(if this limit exists)}.
\]
Similarly, let
\[
\overline{\delta}\lp P\rp:=\limsup_{X\to\infty}
{\frac{\lv\lb p\in\mathcal{P}
\mid p\leq X\text{ and }p\in P\rb\rv}
{\lv\lb p\in\mathcal{P}\mid p\leq X\rb\rv}}.
\]
\begin{impossible}\label{impossible}
We must begin by remarking that it is \emph{a priori} impossible for the cycle statistics of $[f]_\mathcal{P}$ to match those of random dynamical systems.
As shown by Gon\v{c}arov, if $n\in\Z_{\geq1}$ then for any $j\in\Z_{\geq0}$, there is a positive proportion of dynamical systems with precisely $j$ $n$-cycles (that is, $\mu_n(j)>0$ for all $j\in\Z_{\geq0}$).
On the other hand, if $f\in\Z[x]$, $p\in\mathcal{P}$, and $\alpha\in\F_p$ is a point of period $n$ of $\lp\F_p,[f]_p\rp$, then $\alpha$ is a root of $\lc f^n(x)-x\rc_p$; thus, there are no more than $n^{-1}\cdot\deg{\lp f^n\rp}$ $n$-cycles in $\lp\F_p,[f]_p\rp$.
(In fact, something more is true: all points of period $n$ in $\lp\F_p,[f]_p\rp$ are roots of the $n$th dynatomic polynomial of $[f]_p$; we review the theory of dynatomic polynomials in \hyperref[dynatomic]{Section~\ref*{dynatomic}}.)
\end{impossible}
In \hyperref[bigtheorem]{Theorem~\ref*{bigtheorem}} we compute the cycle statistics of the family $\lc x^k+a\rc_\mathcal{P}$ for any $k\in\Z_{\geq2}$ and most $a\in\Z$.
In particular, we obtain the following corollary, which relates these statistics to those discovered by Gon\v{c}arov.

\begin{finalmomentscor}\label{finalmomentscor}
For any $k\in\Z_{\geq2}$ and $n\in\Z_{\geq1}$, there is a probability distribution ${}_k\omega_n\colon\Z_{\geq0}\to\R_{\geq0}$ and a thin set\footnote{
See \hyperref[HITremark]{Remark~\ref*{HITremark}} for a discussion of thin sets and their sizes.} $\mathcal{A}_{k,n}\subseteq\Z$ such that for all $a\in\Z\setminus\mathcal{A}_{k,n}$,
\begin{itemize}
\item
for all $j\in\Z_{\geq0}$,
\[
{}_k\omega_n(j)
=\delta\lp\lb p\in\mathcal{P}
\,\Big\vert\,\lp\F_p,\lc x^k+a\rc_p\rp\text{ has precisely $j$ $n$-cycles}\rb\rp
\]
and
\item
there is an explicit constant $r_k(n)\in\Z_{\geq1}$ such that for all $m\in\Z_{\geq0}$,
\begin{itemize}
\item
if $m\leq r_k(n)$, the $m$th moment of the distributions ${}_k\omega_n$ and $\mu_n$ are the same and
\item
if $m>r_k(n)$, the $m$th moment of ${}_k\omega_n$ is less than the $m$th moment of $\mu_n$.
\end{itemize}
\end{itemize}
\end{finalmomentscor}
\noindent For any $a\in\Z$, the integer $nr_k(n)$ is the degree of the $n$th dynatomic polynomial of $f$, so $r_k(n)$ is the maximum possible number of cycles of length $n$ in $\lp\F_p,\left[x^k+a\right]_p\rp$.
See \hyperref[rkn]{Definition~\ref*{rkn}} for more details.
In particular, we have the inequality $n^{-1}k^{n-1}<r_k(n)<n^{-1}k^n$.

\begin{notsoimpossible}\label{notsoimpossible}
As pointed out in \hyperref[impossible]{Remark~\ref*{impossible}}, for any $k\in\Z_{\geq2}$ and $n\in\Z_{\geq1}$, we know that ${}_k\omega_n(j)=0$ for all $j>r_k(n)$.
Since the matrix $\lp i^j\rp_{1\leq i,j\leq r_k(n)}$ is invertible, we know that there is at most one distribution $\omega\colon\Z_{\geq0}\to\R$ with the property that for all $m\in\lb0,1,\ldots,r_k(n)\rb$, the $m$th moments of $\omega$ and $\mu_n$ are the same.
In other words, \hyperref[finalmomentscor]{Corollary~\ref*{finalmomentscor}} is the statement that for most $a\in\Z$, the cycle distribution of $\lc x^k+a\rc_\mathcal{P}$ is \emph{as random as possible}.
\end{notsoimpossible}

The organization of this paper is as follows: we begin with some background in \hyperref[background]{Section~\ref*{background}}; in particular, we recall Pollard's algorithm and introduce the theory of dynatomic polynomials. 
Afterwards, in in \hyperref[machine]{Section~\ref*{machine}}, we define and study a family of distributions $\lb\omega_{n,r}\rb_{n,r\in\Z_{\geq1}}$ containing the distributions mentioned in \hyperref[finalmomentscor]{Corollary~\ref*{finalmomentscor}} as a subfamily.
In \hyperref[momenttheorem]{Theorem~\ref*{momenttheorem}}, we compute all their moments.
Then, in \hyperref[fixedpointprob]{Theorem~\ref*{fixedpointprob}}, we prove precise estimates for them; this theorem generalizes Theorem~3.5 of~\cite{BG}, which estimates only $\omega_{n,r}(0)$.
In \hyperref[cyclestructuremoments]{Section~\ref*{cyclestructuremoments}} we prove the aforementioned \hyperref[bigtheorem]{Theorem~\ref*{bigtheorem}}, obtaining \hyperref[finalmomentscor]{Corollary~\ref*{finalmomentscor}} as an immediate consequence.
Next, in \hyperref[puttingitalltogether]{Section~\ref*{puttingitalltogether}} we apply our results to prove a version of Conjecture~2.2 of~\cite{MSSS}, which concerns the statistics of $\bigcup_{a\in\Z}{\lc x^2+a\rc_\mathcal{P}}$.
This conjecture asserts in particular that
\[
\lv\lb a\in[p]
\,\Big\vert\,\lp\F_p,\lc x^2+a\rc_p\rp\text{ has precisely one cycle}\rb\rv=O\lp\sqrt{p}\rp.
\]
Our result is \hyperref[thm:final]{Theorem~\ref*{thm:final}}, which we prove in \hyperref[puttingitalltogether]{Section~\ref*{puttingitalltogether}}.

\newtheorem*{thm:final}{\hyperref[thm:final]{Theorem~\ref*{thm:final}}}
\begin{thm:final}
If $k\in\Z_{\geq2}$ and $J\in\Z_{\geq1}$, then for any $\epsilon\in\R_{>0}$, there exists a thin set $\mathcal{A}_{k,J,\epsilon}\subseteq\Z$ such that for all $a\in\Z\setminus\mathcal{A}_{k,J,\epsilon}$,
\[
\overline{\delta}\lp\lb p\in\mathcal{P}
\mid\lp\F_p,\lc x^k+a\rc_p\rp\text{ has }J\text{ or fewer cycles}\rb\rp
<\epsilon.
\]
\end{thm:final}
\noindent Note that \hyperref[thm:final]{Theorem~\ref*{thm:final}} addresses the presence of arbitrarily many cycles, for fixed unicritical polynomials of arbitrarily large degree.

Additionally, in \hyperref[cyclestructure]{Section~\ref*{cyclestructure}} we introduce a family of cycle densities on the set of \emph{all} monic binomial unicritical polynomials in $\Z[x]$.
In \hyperref[cycledistribution]{Theorem~\ref*{cycledistribution}} and \hyperref[disconnectedbox]{Corollary~\ref*{disconnectedbox}}, we prove in particular that
\begin{itemize}
\item
for all $k\in\Z_{\geq2}$ and $n\in\Z_{\geq1}$
\[
\lim_{X\to\infty}
{\frac{1}{2X+1}
\sum_{a=-X}^X
{\delta\lp\lb p\in\mathcal{P}\,\Big\vert\,
\lp\F_p,\lc x^k+a\rc_p\rp\text{ has precisely $j$ $n$-cycles}\rb\rp}}
={}_k\omega_n(j),
\]
and
\item
for all $k\in\Z_{\geq2}$, $J\in\Z_{\geq1}$, and $\epsilon\in\R_{>0}$, there exists $N\in\Z_{\geq1}$ such that for all $n\in\Z_{\geq N}$,
\[
\lim_{X\to\infty}
{\frac{1}{2X+1}
{\sum_{a=-X}^X
{\delta\lp\lb p\in\mathcal{P}\,\Big\vert\,
\substack{\lp\F_p,\lc x^k+a\rc_p\rp\text{ has $J$ or fewer cycles}\\\text{of length at most $n$}}\rb\rp}}}<\epsilon.
\]
\end{itemize}
The latter fact implies that there is an increasing function $\gamma\colon\Z_{\geq1}\to\Z_{\geq1}$ such that
\[
\lim_{J\to\infty}
{\frac{1}{2\gamma(J)+1}
{\sum_{a=-\gamma(J)}^{\gamma(J)}
{\delta\lp\lb p\in\mathcal{P}\,\Big\vert\,
\substack{\lp\F_p,\lc x^k+a\rc_p\rp\text{ has $J$ or fewer cycles}\\\text{of length at most $\gamma(J)$}}\rb\rp}}}=0.
\]
We conclude the paper with a conjecture on the growth rate of $\gamma$.


\section{Background}
\label{background}

\subsection{Pollard's rho algorithm and randomness}
\label{pollardrho}

In~\cite{Pollard}, Pollard presented his famous ``rho algorithm" for integer factorization, which was the first algorithm to factor the Fermat number $2^{256}+1$, see~\cite{BP}.
A conjectural termination time of this algorithm relies upon an aspect of the supposed ``randomness'' of $\lc x^2+1\rc_\mathcal{P}$.
To better state this conjecture, we introduce a bit more notation.
For any dynamical system $(S,f)$ and $s\in S$, we let
\[
\mathbf{O}_{(S,f)}(s)=\lv\lb f^n(s)\mid n\in\Z_{\geq0}\rb\rv.
\]
(If $\lb f^n(s)\mid n\in\Z_{\geq0}\rb$ is infinite, we write $\mathbf{O}_{(S,f)}(s)=\infty$.)
Now, letting $g=x^2+1$, then the conjecture that
\[
\mathbf{O}_{\lp\F_p,[g]_p\rp}(0)
=O\lp\sqrt{p}\rp
\]
implies that there is some $C\in\R_{>0}$ such that Pollard's algorithm, applied to any positive composite integer $m$, terminates in at most
\[
C\sqrt{\min{\lp \lb p\in\mathcal{P}
\mid p\text{ divides }m\rb\rp}}
\]
steps.
This conjecture formalizes the hope that for all $p\in\mathcal{P}$, the ``time to first repeat'' of $[g]_p$ (acting on $0\in\F_p$) is the same as that of a random self-map on a set of size $p$, which is well-known and classical; indeed,
\[
\frac{1}{\lv[X]^{[X]}\rv}\sum_{f\in[X]^{[X]}}
{\mathbf{O}_{\lp[X],f\rp}(0)}
=O\lp\sqrt{X}\rp.
\]
This conjecture provides motivation for the question mentioned in \hyperref[intro]{Section~\ref*{intro}}: what properties do families of polynomials share with random maps?
See \cite{BGHKST,SilvermanVariationmodp} 
for further analysis of the ``random map" 
model in arithmetic dynamics and its limitations, and~\cite{FlajoletOdlyzko} for an extensive presentation of the statistics of random mappings (also~\cite{Harris,Stepanov,ArneyBender}).

There has been quite a bit of recent work on this question, both for the families
\[
[f]_{\mathcal{P}},
\hspace{10px}\text{for $f\in\Z[x]$},
\]
and the families
\[
\bigcup_{\substack{f\in\Z[x]\\\deg{f}=k}}{[f]_{\mathcal{P}}},
\hspace{10px}\text{for $k\in\Z_{\geq2}$.}
\]
For the latter case, see \cite{FG}, \cite{BS}, \cite{BGTW}, and in particular~\cite{Bach},
who proved that for any $N\in\Z_{\geq1}$, there is a constant $C_N\in\R_{>0}$ with the property that for all $p\in\mathcal{P}$,
\[
\frac{1}{p^2}\sum_{\alpha\in\F_p}
{\lv\lb\beta\in\F_p\mid
\mathbf{O}_{\lp\F_p,x^2+\alpha\rp}(\beta)\leq N\rb\rv}
>\frac{1}{p}\binom{N}{2}+C_Np^{-\frac{3}{2}}.
\]
In the former case, Juul, Kurlberg, Madhu, and Tucker prove in~\cite{JKMT} that if $f=x^k+a$, with $k\in\Z_{\geq2}$ and $a\in\Z$, and $f$ is not conjugate to the Chebyshev polynomial $x^2-2$, then
\[
\liminf_{p\to\infty}{\lp\frac{1}{p}\cdot\lv\lb\alpha\in\F_p\mid \alpha\text{ periodic in }\lp\F_p,[f]_p\rp\rb\rv\rp}=0.
\]
In~\cite{Heath-Brown}, Heath-Brown computed explicit bounds on periodic points of a family of dynamical systems; indeed, for any $a,c\in\Z_{\geq1}$, he showed that
\[
\lv\lb\alpha\in\F_p\mid \alpha\text{ periodic in }\lp\F_p,[ax^2+c]_p\rp\rb\rv=O_{a,c}\lp\frac{p}{\log{\log{p}}}\rp.
\]

Both Bach and Heath-Brown obtain their results by using classical point-counting techniques (such as Weil's ``Riemann Hypothesis" and Bezout's Theorem) to study the curves $\lp[f]_p\rp^n(x)=\lp[f]_p\rp^n(y)$, for $f\in\Z[x]$ and $n\in\Z_{\geq1}$.
On the other hand, Juul et.\ al.\ use the ``arboreal Galois theory" pioneered by Odoni~\cite{Odoni}, and elaborated upon by Boston, Jones, and many others~\cite{BostonJones}.
This theory analyzes Galois groups of polynomials of the form $f^n(x)-a$ for $f\in\Q[x]$ and $a\in\Q$.

To analyze the cycle structure of periodic points of dynamical system, we use the Galois theory of dynatomic polynomials, which we now recall.

\subsection{Dynatomic polynomials}
\label{dynatomic}

As we intend to study the cycle structure of dynamical systems induced by polynomials, we will make use of the theory of dynatomic polynomials (and their Galois groups).
See \cite{MP}, \cite{MorCurves} (and the correction in \cite{MorCorrection}), \cite{MorGalGroups}, and ~\cite[Chapter 4.1]{SilvermanADS} for background in this area.
We sketch an introduction, focusing on the aspects of the theory we will use in our results.

Let $K$ be a field, $f\in K[x]$, and $n\in\Z_{>0}$.
The points of period $n$ of the dynamical system $\lp K,f\rp$ are certainly roots of the polynomial $f^n(x)-x$.
However, if $d\in\Z_{\geq1}$ and $d\mid n$, then this polynomial vanishes on points of period $d$ as well (for example, if $\alpha\in K$ is a fixed point of $(K,f)$, i.e. $f(\alpha)=\alpha$, then $f^n(\alpha)=\alpha$ for all $n\in\Z_{\geq1}$).
In an attempt to sieve out the points of lower period, one defines the $n$th \emph{dynatomic polynomial of $f$} for any $n\in\Z_{\geq1}$:
\[
\Phi_{f,n}(x)
:=\prod_{d\mid n}
{\lp f^d(x)-x\rp^{\mu(n/d)}},
\]
where $\mu\colon\Z_{\geq0}\to\lb-1,0,1\rb$ is the usual M\"obius function.
The fact that 
\[
\prod_{d\mid n} \Phi_{f,n}(x) = f^n(x)-x
\]
follows quickly by applying the M\"obius inversion formula.
As usual, we omit ``$K$'' from the notation ``$\Phi_{f,n}$''; we will always specify the set of coefficients of $f$, so that the field $K$ will be clear from context.
As indicated by its name, the $n$th dynatomic polynomial is analogous to the $n$th cyclotomic polynomial, which vanishes precisely on primitive $n$th roots of unity.
(It turns out that $\Phi_{f,n}$ may occasionally vanish on points of period $d$ for $d<n$: see~\cite[Example 4.2]{SilvermanADS}.
Luckily, \hyperref[rootsandcycles]{Proposition~\ref*{rootsandcycles}} addresses this inconvenience.)
We should mention that it is not \emph{a priori} obvious that $\Phi_{f,n}$ is a polynomial.
See~\cite[Theorem~2.5]{MP} for a proof that $\Phi_{f,n}\in K[x]$.
(In particular, if $f\in\Z[x]$ and $f$ is monic, then $\Phi_{f,n}\in\Z[x]$ by Gauss's Lemma.)
The degrees of certain dynatomic polynomials will be important quantities in many computations that follow, so we introduce the following notation.
\begin{rkn}\label{rkn}
For any $n\in\Z_{\geq1}$ and $k\in\Z_{\geq2}$, define $r_k(n)$ to be the positive integer such that $n\cdot r_k(n)$ is the degree (in $x$) of the $n$th dynatomic polynomial of $x^k+c\in\Q(c)[x]$;
that is,
\[
r_k(n)=\frac{1}{n}\cdot
\sum_{d\mid n}
{k^d\mu\lp\frac{n}{d}\rp}.
\]
\end{rkn}
\begin{rknfax}\label{rknfax}
For any $n\in\Z_{\geq1}$ and $k\in\Z_{\geq2}$, the quantity $r_k(n)$ is quite large compared to $n$; indeed, for all such $n$ and $k$,
\[
\frac{k^{n-1}}{n}<r_k(n)<\frac{k^n}{n}.
\]
\end{rknfax}

Our proofs of \hyperref[thm:final]{Theorem~\ref*{thm:final}} and \hyperref[bigtheorem]{Theorem~\ref*{bigtheorem}} rely in part on the knowledge of the structure of the Galois groups of $\Phi_{f,n}$, where $n\in\Z_{\geq1}$ and $f(x)=x^k+a\in\Z[x]$ for $k\in\Z_{\geq2}$ and $a\in\Z$.
For a specific polynomial $f\in\Z[x]$ of this form and any large $n$, it is difficult to compute the Galois group of $\Phi_{f,n}$, since the degree of $\Phi_{f,n}$ is so large, but---thanks to work of Morton~\cite{MorGalGroups}---the Galois groups of $\Phi_{f,n}$ for $f(x)=x^k+c\in\Q(c)[x]$ are known.
To state \ref{mortonpower}, we introduce a bit more notation.
If $f=f(x)=x^k+c\in\Q(c)[x]$ and $n\in\Z_{\geq1}$, let
\begin{itemize}
\item $\Sigma_{f,n}$ denote the splitting field of $\Phi_{f,n}$ over $\Q(c)$, and
\item $K_{f,n}$ denote the splitting field of $f^n(x)-x$ over $\Q(c)$.
\end{itemize}
Note that $K_{f,n}$ is the compositum of the fields in $\lb\Sigma_{f,d}\mid d\in\Z_{\geq1}\text{ and }d\mid n\rb$. 

We now have most of the notation to state \ref{mortonpower}.
We postpone the discussion of wreath products (which appear in \ref{mortonpower} as $C_d\wr S_{r_k(d)}$) and their natural action (on the sets $B\lp n,r_k(n)\rp$) until \hyperref[machine]{Section~\ref*{machine}}, which immediately follows the statement of the theorem; there, we will introduce wreath products and their actions, then discuss them in detail.
The following theorem combines results from~\cite{MorGalGroups}.
\begin{named}{Morton's Theorem}\label{mortonpower}
Let $k\in\Z_{\geq2}$, $n\in\Z_{\geq1}$, and $f=f(x)=x^k+c\in\Q(c)[x]$.
Next, set $\mathcal{S}
=\lb\Sigma_{f,d}
\mid d\in\Z_{\geq1}\text{ and }d\mid n\rb$.
Then $f^n(x)-x$ has no repeated roots, and if $\Sigma_{f,d}\in\mathcal{S}$, then
\begin{itemize}
\item
the field $\Sigma_{f,d}$ is linearly disjoint from the compositum of the fields in $\mathcal{S}\setminus\lb \Sigma_{f,d}\rb$,
\item
the roots of $\Phi_{f,d}$ are precisely the points of period $d$ in the dynamical system $\lp\Sigma_{f,d},f\rp$,
\item
$\Gal{\lp \Sigma_{f,d}/\Q(c)\rp}\simeq C_d\wr S_{r_k(d)}$, and
\item
the action of $\Gal{\lp \Sigma_{f,d}/\Q(c)\rp}$ on the points of period $d$ in the dynamical system $\lp\Sigma_{f,d},f\rp$ matches the action of $C_d\wr S_{r_k(d)}$ on $B(d,r_k(d))$.
\end{itemize}
\end{named}

\section{Fixed points in wreath product actions}
\label{machine}

In this section, we study the statistics of fixed-point proportions in a family of wreath products. 
This family includes the groups that appear as Galois groups of dynatomic polynomials, so 
these statistics are a vital component of our proofs of \hyperref[bigtheorem]{Theorem~\ref*{bigtheorem}} and \hyperref[thm:final]{Theorem~\ref*{thm:final}}.
We begin with some definitions.

Suppose that $r\in\Z_{\geq1}$.
For any group $G$, we write $G\wr S_r$ to mean the wreath product $G\wr_{\{1,\dots,r\}} S_r$.
That is, $G\wr S_r=G^r\rtimes S_r$, where $S_r$ acts on $G^r$ by permuting coordinates.
In particular, we note that $\lv G\wr S_r\rv=r!|G|^r$. See~\cite[Chapter 3A]{IsaacsGroupTheory} for background on wreath products.
For any $n\in\Z_{\geq1}$, we let $C_n$ denote the cyclic group of size $n$; we then write $B(n,r)$ for the set $C_n\times\lb1,\ldots,r\rb$.
The group $C_n\wr S_r$ acts naturally on $B(n,r)$; concretely, for any $\sigma=\lp\lp\zeta_1,\ldots,\zeta_r\rp,\pi\rp\in C_n\wr S_r$, this action is given by
\begin{align*}
\sigma\colon B(n,r)&\to B(n,r)\\
\lp\zeta,i\rp&\mapsto\lp\zeta_i\cdot\zeta,\pi(i)\rp.
\end{align*}
For any $\sigma\in C_n\wr S_r$, define
\[
\Fix{\sigma}=\lb\lp\zeta,i\rp\in B(n,r)
\,\,\big\vert\,\,\sigma\lp\zeta,i\rp=\lp\zeta,i\rp\rb.
\]
Observe that $n$ divides $\lv\Fix{\sigma}\rv$---this follows from the fact that if $\sigma$ fixes any $\lp\zeta_0,i\rp\in B(n,r)$, then it must fix each $\lp\zeta,i\rp$ for all $\zeta\in C_n$.
With this fact in mind, we now define the random variable $\mathbf{W}_{n,r}$ on $C_n\wr S_{r}$ by
\begin{align*}
\mathbf{W}_{n,r}\colon C_n\wr S_r
&\to\Z_{\geq0}\\
\sigma
&\mapsto\frac{1}{n}\cdot
\lv\Fix{\sigma}\rv.
\end{align*}
Note that the random variable $n\cdot\mathbf{W}_{n,r}$ is the permutation character of the action of $C_n\wr S_r$ on $B(n,r)$.
We will denote the probability distribution associated to $\mathbf{W}_{n,r}$ by $\omega_{n,r}$; concretely,
\begin{align*}
\omega_{n,r}\colon\Z_{\geq0}
&\to\R_{\geq0}\\
j&\mapsto\frac{\lv\lb\sigma\in C_n\wr S_r
\,\big\vert\,\lv\Fix{\sigma}\rv=jn\rb\rv}{\lv C_n\wr S_r\rv}.
\end{align*}
\hyperref[moments]{Section~\ref*{moments}} is devoted to computing the moments of $\mathbf{W}_{n,r}$.
Then, \hyperref[probbounds]{Section~\ref*{probbounds}} computes useful bounds on $\omega_{n,r}$.
These bounds generalize Theorem~3.5 of~\cite{BG}; in \hyperref[puttingitalltogether]{Section~\ref*{puttingitalltogether}} we apply these bounds to prove \hyperref[thm:final]{Theorem~\ref*{thm:final}}.


\subsection{Moments of permutation characters}
\label{moments}

In this section, we compute the moments of the distributions $\omega_{n,r}$.
Recall that if $\omega$ is a discrete probability distribution on $\Z_{\geq0}$, the \emph{$m$th moment} of $\omega$ is
\[
\sum_{j=0}^\infty{j^m\omega(j)};
\]
we denote the $m$th moment of a probability distribution $\omega$ by $\M_m{(\omega)}$.
If $\mathbf{X}$ is a random variable with codomain $\Z_{\geq0}$, we will write $\M_m{(\mathbf{X})}$ for the moment of the probability distribution associated to $\mathbf{X}$.
The moments of the Poisson distribution are well known; they are expressed in terms of Stirling numbers of the second kind.
For any $m,i\in\Z_{\geq0}$, the \emph{Stirling number of the second kind} is the number of ways of partitioning a set of $m$ elements into $i$ distinct nonempty subsets; we denote this number by $\lb m\atop i\rb$.
By convention we take
\[
\lb m\atop 0\rb=
\begin{cases}
1&\text{ if }m=0,\\
0&\text{ otherwise.}
\end{cases}
\]
\begin{pmoments}\label{pmoments}
For all $\lambda\in\R_{>0}$ and $m\in\Z_{\geq0}$,
\[
\M_m{\lp\rho_\lambda\rp}
=\sum_{i=0}^m\left\{m\atop i\right\}\lambda^i.
\]
\end{pmoments}
We now turn to computing the moments of $n\cdot\mathbf{W}_{n,r}$ for any $n,r\in\Z_{\geq0}$; before proceeding, we prove a combinatorial identity involving Stirling numbers.

%


\begin{stirlinglemma}\label{stirlinglemma}
If $m\in\Z_{\geq1}$ and $i\in\Z_{\geq0}$, then
\[
\sum_{\ell=i}^{m-1}
{\binom{m-1}{\ell}\lb\ell\atop i\rb}
=\lb m\atop i+1\rb.
\]
\end{stirlinglemma}
\begin{proof}
For any $\ell\in\{1,\dots,m-1\}$, the number of ways to choose a subset of $\{1,\dots,m-1\}$ of size $\ell$ and partition it into $i$ nonempty subsets is $\binom{m-1}{\ell}\lb\ell\atop i\rb$.
These choices are in bijection with partitions of $\{1,\dots,m\}$ into $i+1$ nonempty subsets with the condition that the subset containing $m$ is of size $m-\ell$.
Summing over $\ell$ yields the result.
\end{proof}

We now have the tools to use character theory to compute the moments of $n\cdot\mathbf{W}_{n,r}$; the moments of $\mathbf{W}_{n,r}$ will follow as an immediate corollary.

\begin{momenttheorem}\label{momenttheorem}
Let $m\in\Z_{\geq0}$. 
For all $n,r\in\Z_{\geq1}$,
\[
\M_m{\lp n\cdot\mathbf{W}_{n,r}\rp}
=\sum_{i=0}^{\min{\lp\lb m,r\rb\rp}}
{\lb m\atop i\rb n^{m-i}}
\]
\end{momenttheorem}
\begin{proof}
We induct on $m$. The $m=0$ case is trivial and the $m=1$ case follows immediately from Burnside's Lemma: indeed, $C_n\wr S_r$ acts transitively on $B(n,r)$, so 
\[
\frac{1}{\lv C_n\wr S_r\rv}\cdot
\sum_{\sigma\in C_n\wr S_r}n\cdot\mathbf{W}_{n,r}(\sigma) = 1=\left\{1\atop 1\right\}n^{0}.
\]

Now assume that $m\geq 2$ and that the statement of the theorem is true for every $\ell\in\lb0,\ldots,m-1\rb$.
Choose any $n,r\in\Z_{\geq1}$ and any $\alpha\in B(n,r)$.
To ease notation, let
\begin{itemize}
\item
$W=C_n\wr S_r$,
\item
$\chi=n\cdot\mathbf{W}_{n,r}$,
\item
$W_\alpha=\lb\sigma\in W\mid\alpha\in\Fix{\sigma}\rb$, the isotropy subgroup of $\alpha$,
\item
$\chi_\alpha$ be the restriction of $\chi$ to $W_\alpha$, and
\item
$\mathbbm{1}_\alpha$ be the principal character of $W_\alpha$.
\end{itemize}
Additionally, let $\langle\cdot,\cdot\rangle$ be the usual inner product of class functions.
Since $\chi$ is the permutation character of the action of $W$ on $B(n,r)$, we know that $\chi$ is 
induced from $\mathbbm{1}_\alpha$; see~\cite[Lemma 5.14]{IsaacsCharacterTheory}.
That is, $\chi =\lp \mathbbm{1}_\alpha\rp^W$.
By Frobenius reciprocity, we compute
\[
\M_m{\lp n\cdot\mathbf{W}_{n,r}\rp}
=\frac{1}{|W|}\sum_{\sigma\in W}\chi^m(\sigma)
=\la\chi^{m-1},\chi\ra
=\la\chi^{m-1},\lp\mathbbm{1}_\alpha\rp^W\ra
=\la\lp\chi_\alpha\rp^{m-1},\mathbbm{1}_\alpha\ra.
\]
Observe that if $\alpha=(\zeta,i)$, then $W_\alpha$ acts trivially on the $n$ elements of $C_n\times\{i\}$ and acts transitively on the remaining elements in $B(n,r)$, if there are any.
There are two cases.
\begin{itemize}
\item
If $r\geq2$, the preceding observation implies both that we can restrict the action of $W_\alpha$ to an action on $B(n,r)\setminus\lp C_n\times\{i\}\rp$ as well as that $W_\alpha\simeq C_n\wr S_{r-1}$.
If we let $\psi$ be the permutation character of this restricted action, then the induction hypothesis implies that for all $\ell\in\lb0,\ldots, m-1\rb$,
\[
\frac{1}{\lv W_\alpha\rv}\sum_{\sigma\in W_\alpha}{\psi^\ell(\sigma)}
=\M_\ell{\lp n\cdot\mathbf{W}_{n,r-1}\rp}
=\sum_{i=0}^{\min{\lp\lb r-1,\ell\rb\rp}}{\lb \ell\atop i\rb n^{\ell-i}}.
\]
\item
On the other hand, if we are in the case where $r=1$, then $W_\alpha$ is the trivial group.
In this case, we let $\psi$ be the trivial class function of $W_\alpha$---that is, $\psi$ evaluates to zero on the single element of $W_\alpha$.
We adopt the convention that $\psi^0=\mathbbm{1}_\alpha$, so that for all $\ell\in\lb0,\ldots, m-1\rb$,
\[
\frac{1}{\lv W_\alpha\rv}\sum_{\sigma\in W_\alpha}{\psi^\ell(\sigma)}
=\sum_{i=0}^{\min{\lp\lb r-1,\ell\rb\rp}}{\lb\ell\atop i\rb n^{\ell-i}},
\]
as in the previous case.
\end{itemize}
In either case, note that $\chi_\alpha=n\mathbbm{1}_\alpha+\psi$.
Thus,
\begin{align*}
\la\lp\chi_\alpha\rp^{m-1},\mathbbm{1}_\alpha\ra
&=\la\lp n\mathbbm{1}_\alpha+\psi\rp^{m-1},\mathbbm{1}_\alpha\ra&&\\
&=\sum_{\ell=0}^{m-1}
{\binom{m-1}{\ell}n^{m-1-\ell}\la\mathbbm{1}_\alpha,\psi^\ell\ra}&&\text{since }\psi^0=\mathbbm{1}_\alpha\\
&=\sum_{\ell=0}^{m-1}
{\binom{m-1}{\ell}n^{m-1-\ell}\lp\frac{1}{\lv W_\alpha\rv}
\sum_{\sigma\in W_\alpha}
{\psi^\ell(\sigma)}\rp}
&&\text{by definition of }\la\cdot,\cdot\ra\\
&=\sum_{\ell=0}^{m-1}
{\binom{m-1}{\ell}n^{m-1-\ell}
\sum_{i=0}^{\min{\lp\lb r-1,\ell\rb\rp}}
{\lb\ell\atop i\rb n^{\ell-i}}}
&&\text{by the induction hypothesis}\\
&=\sum_{i=0}^{\min{\lp\lb r-1,m-1\rb\rp}}
{n^{m-1-i}\sum_{\ell=i}^{m-1}
{\binom{m-1}{\ell}\lb \ell\atop i\rb}}\\
&=\sum_{i=0}^{\min{\lp\lb r-1,m-1\rb\rp}}
{n^{m-1-i}\lb m\atop i+1\rb}
&&\text{by \hyperref[stirlinglemma]{Lemma~\ref*{stirlinglemma}}}\\
& = \sum_{i=0}^{\min{\lp\lb r,m\rb\rp}}\left\{m\atop i\right\}n^{m-i},
\end{align*}
completing the proof.
\end{proof}

\begin{momentswecareabout}\label{momentswecareabout}
If $n,r\in\Z_{\geq1}$, then for all $m\in\Z_{\geq0}$,
\[
\M_m{\lp\mathbf{W}_{n,r}\rp}
=\sum_{i=0}^{\min{\lp\lb m,r\rb\rp}}{\lb m\atop i\rb n^{-i}}.
\]
In particular, if $m\leq r$, then
\[
\M_m{\lp\omega_{n,r}\rp}
=\M_m{\lp\rho_{1/n}\rp}.
\]
\end{momentswecareabout}
\begin{proof}
Immediate from \hyperref[momenttheorem]{Theorem~\ref*{momenttheorem}} and \hyperref[pmoments]{Fact~\ref*{pmoments}}.
\end{proof}

\begin{momentremark}
In~\cite{ChowMansour}, the authors also studied the moments of the distributions $\omega_{n,r}$ in the context of an application of their main results on moment generating functions.
However, their computation of the moments of $\omega_{n,r}$ contains an error: they state that $\M_m\lp\omega_{n,r}\rp=\M_m{\lp n\cdot\mathbf{W}_{n,r}\rp}
=\sum_{i=0}^r
{\lb m\atop i\rb n^{m-i}}$, but this equality only holds when $m\leq r$. 
Our methods are entirely different, relying instead on character theory.
%
\end{momentremark}



\subsection{Bounds on fixed point probabilities in wreath products}
\label{probbounds}

Having computed the moments of $\omega_{n,r}$ for all $n,r\in\Z_{\geq1}$, we now bound $\omega_{n,r}(j)$ for any $j\in\Z_{\geq 0}$.
Before proceeding, we introduce a bit of notation.
For any $r\in\Z_{\geq1}$ and $i\in\lb0,\ldots,r\rb$, let $D_{r,i}$ be the $(r,i)$th rencontres number; i.e. the number of permutations in $S_r$ with exactly $i$ fixed points.
For convenience, we set $D_{0,0}=1$.
Note that $D_{r,0}$ is the number of derangements in $S_r$.

\begin{fixedpointprob}\label{fixedpointprob}
Let $k\in\Z_{\geq2}$ and $r,n\in\Z_{\geq1}$.
For any $j\in\Z_{\geq0}$, if $j\leq r$, then
\[
\lv\omega_{n,r}(j)
-\frac{e^{-1/n}}{j!n^j}\rv
<\frac{1+2^{r-j}}{j!n^j(r-j)!}.
\]
If $j>r$, then $\omega_{n,r}(j)=0$.
\end{fixedpointprob}
\begin{proof}
If $j>r$, the result is trivial, so suppose that $j\leq r$.
Choose $\sigma=\lp\lp\zeta_1,\ldots,\zeta_r\rp,\pi\rp\in C_n\wr S_r$. 
Note that if $\lv\Fix{\sigma}\rv=nj$, then $\pi$---acting on $\lb1,\ldots,r\rb$---has at least $j$ fixed points.
Moreover, there is a subset $R$ of the fixed points of $\pi$ such that
\begin{itemize}
\item $\lv R\rv=j$ and
\item if $i^\prime\in\lb1,\ldots,r\rb$ is a fixed point of $\pi$, then $i^\prime\in R$ if and only if $\zeta_{i^\prime}=1$.
\end{itemize}
Using this fact, and enumerating permutations $\pi$ by their number of fixed points, we conclude that
\[
\lv\lb\sigma\in C_n\wr S_r\mid\lv\Fix{\sigma}\rv=nj\rb\rv
=\sum_{i=0}^r{\binom{i}{j}D_{r,i}(n-1)^{i-j}n^{r-i}}.
\]
Now, for any $i\in\Z_{\geq0}$ it is simple to show $D_{r,i}={\binom{r}{i}}D_{r-i,0}$ (see~\cite[Lemma 3.4]{BG}), so we see
\[
\omega_{n,r}(j)
=\frac{1}{r!n^r}\sum_{i=j}^r{\binom{i}{j}\binom{r}{i}D_{r-i,0}(n-1)^{i-j}n^{r-i}}.
\]
Next, recall the well-known fact that for any $i\in\Z_{\geq0}$, the $(i,0)$th rencontres number $D_{i,0}$ satisfies $\frac{i!}{e}-1<D_{i,0}<\frac{i!}{e}+1$.
Thus,
\begin{align*}
\lv\omega_{n,r}(j)
-\frac{1}{r!n^je}\sum_{i=j}^r{\binom{i}{j}\binom{r}{i}(r-i)!\lp\frac{n-1}{n}\rp^{i-j}}\rv
<\frac{1}{r!n^j}\sum_{i=j}^r{\binom{i}{j}\binom{r}{i}\lp\frac{n-1}{n}\rp^{i-j}}.
\end{align*}
We will address the approximation and error terms in turn.
\begin{itemize}
\item
Approximating the Taylor series remainder of the approximation, we see that
\begin{align*}
\frac{1}{r!n^je}\sum_{i=j}^r{\binom{i}{j}\binom{r}{i}(r-i)!\lp\frac{n-1}{n}\rp^{i-j}}
&=\frac{1}{j!n^je}\sum_{i=j}^r{\frac{1}{(i-j)!}\lp\frac{n-1}{n}\rp^{i-j}}\\
&=\frac{1}{j!n^je}\sum_{i=0}^{r-j}{\frac{1}{i!}\lp\frac{n-1}{n}\rp^i}\\
&<\frac{e^{-1/n}}{j!n^j}+\frac{1}{j!n^je\cdot(r-j+1)!}.
\end{align*}
\item
As for the error term, we certainly know that
\[
\lp1+\frac{n-1}{n}\rp^{r-j}<2^{r-j},
\]
so the binomial theorem implies
\begin{align*}
\frac{1}{r!n^j}\sum_{i=j}^r{\binom{i}{j}\binom{r}{i}\lp\frac{n-1}{n}\rp^{i-j}}
&=\frac{1}{j!n^j}\sum_{i=j}^r{\frac{1}{(i-j)!(r-i)!}\lp\frac{n-1}{n}\rp^{i-j}}\\
&=\frac{1}{j!n^j(r-j)!}\sum_{i=0}^{r-j}{\binom{r-j}{i}\lp\frac{n-1}{n}\rp^i}\\
&< \frac{2^{r-j}}{j!n^j(r-j)!}.
\end{align*}
\end{itemize}
Thus, the theorem is true.
\end{proof}

For our applications, we must estimate those wreath products occurring as Galois groups of dynatomic polynomials, so we introduce the following notation.
\begin{komegan}
For any $k\in\Z_{\geq2}$ and $n\in\Z_{\geq1}$, let ${}_k\omega_n := \omega_{n,r_k(n)}.$
(Recall that $nr_k(n)$ is the degree of the $n$th dynatomic polynomial of $x^k+c\in\Q(c)[x]$, see \hyperref[rkn]{Definition~\ref*{rkn}}.)
\end{komegan}
\noindent For any $k\in\Z_{\geq2}$ and $j\in\Z_{\geq0}$, the sequence $\lb{}_k\omega_n(j)\rb_{n\in\Z{\geq1}}$ has particularly stable behavior as $n\to\infty$; we now record this behavior for use in the proof of \hyperref[thm:final]{Theorem~\ref*{thm:final}} in \hyperref[puttingitalltogether]{Section~\ref*{puttingitalltogether}}.

\begin{fixedpointprobcor}\label{fixedpointprobcor}
If $k\in\Z_{\geq2}$ and $j\in\Z_{\geq0}$, then
\[
{}_k\omega_n(j)
=\frac{e^{-1/n}}{j!n^j}
+O\lp\frac{2^{r_k(n)-j}}{\lp r_k(n)-j\rp!}\rp.
\]
In particular,
\[
{}_k\omega_n(j)
=\frac{1}{j!n^j}\lp1-\frac{1}{n}\rp+O\lp\frac{1}{n^{j+2}}\rp.
\]

\end{fixedpointprobcor}
\begin{proof}
This follows immediately from \hyperref[fixedpointprob]{Theorem~\ref*{fixedpointprob}}, \hyperref[rknfax]{Remark~\ref*{rknfax}}, and Taylor's Theorem.
\end{proof}



\section{The cycle structure of unicritical polynomials}
\label{cyclestructure}

Before proving \hyperref[bigtheorem]{Theorem~\ref*{bigtheorem}} and \hyperref[thm:final]{Theorem~\ref*{thm:final}}, we recall three important results, in the forms which will be most useful for the remainder of the paper.
In this section, for any polynomial $f=f(c,x)\in\Q[c][x]$ and any $a\in\Q$, we will write $f_a$ for the specialization of $f$ at $c=a$; that is, $f_a=f_a(x)=f(a,x)\in\Q[x]$.

The first result addresses the following inconvenience: for any $f\in\Z[x]$, $n\in\Z_{\geq1}$, and $p\in\mathcal{P}$, the polynomial $\lc\Phi_{f,n}\rc_p$ certainly vanishes on the period $n$ points of $\lp\F_p,[f]_p\rp$, but, as discussed in~\cite{BG}, it occasionally vanishes at points of lower period as well (indeed this can happen even in characteristic $0$---see \cite[Example 4.2]{SilvermanADS}).
Luckily, as long as $f^n(x)-x$ has distinct roots, this pathology can only occur for finitely many $p\in\mathcal{P}$.
We record this fact as \hyperref[rootsandcycles]{Proposition~\ref{rootsandcycles}}. 
\begin{rootsandcycles}\label{rootsandcycles}
Let $f\in\Z[x]$ and $n\in\Z_{\geq1}$, and suppose that $f^n(x)-x$ has no repeated roots.
If $j\in\Z_{\geq0}$, then for all but finitely many $p\in\mathcal{P}$,
\[
\lp\F_p,[f]_p\rp\text{ has precisely $j$ $n$-cycles}
\hspace{10px}\text{ if and only if }\hspace{10px}
\lc\Phi_{f,n}\rc_p\text{ has precisely $jn$ roots in }\F_p.
\]
\end{rootsandcycles}
\begin{proof}
This is a trivial generalization of \cite[Corollary~4.3]{BG}, which follows from \cite[Theorem 4.5]{SilvermanADS}.
\end{proof}

Next, we state the forms of the \ref{HIT} and \ref{FDT} which will be most useful to us in the following sections.

\begin{named}{Hilbert Irreducibility Theorem}\label{HIT}
Let $f(c,x)\in\Z[c][x]$, let $K$ be the splitting field of $f(c,x)$ over $\Q(c)$, and for any $a\in\Z$, let $K_a$ be the splitting field of $f_a$ over $\Q$.
Suppose that $f(c,x)$ has no repeated roots \textup{(}in $\overline{\Q(c)}$\textup{)}. Then there exists a ``thin set'' $\mathcal{A}\subset\Z$ such that for all $a\in\Z\setminus\mathcal{A}$,
\[
f_a\text{ has no repeated roots}
\hspace{10px}\text{and}\hspace{10px}
\Gal{\lp K_a/\Q\rp}\simeq\Gal{\lp K/\Q(c)\rp}.
\]
\end{named}
\begin{HITremark}\label{HITremark}
For details on the connection between the Hilbert Irreducibility Theorem and Galois theory, see, for example,~\cite{CohenHIT},~\cite[Chapter VIII]{LangDiophantineGeometry}, and~\cite[Chapter 1]{Volklein}.

As for the size of ``thin'' sets, for any thin set $\mathcal{A}$ there is some constant $C_\mathcal{A}\in\R_{>0}$ such that for all $X\in\Z_{\geq0}$,
\[
\lv\lb a\in\mathcal{A}\mid-X\leq a\leq X\rb\rv\leq C_\mathcal{A}\sqrt{X}.
\]
(See~\cite{Serre}, Section~9.7, for more details).
\end{HITremark}


\begin{named}{Frobenius Density Theorem}\label{FDT}
%
%
Suppose that $f(x)\in\Z[x]$ is a monic polynomial with no repeated roots, and fix any $j\in\Z_{\geq0}$.
Let $G=\Gal{(f/\Q)}$ and $P\subseteq\mathcal{P}$ be the set of primes $p$ such that $[f]_p$ has exactly $j$ roots in $\F_p$. Then
\[
\delta(P)
=\frac{1}{\lv G\rv}\cdot
\lv\lb\sigma\in G
\mid
\sigma\text{ fixes exactly $j$ roots of $f$}
\rb\rv.
\]
\end{named}
\noindent (See~\cite{SL} for more details on this theorem.)

\begin{FDTremark}\label{FDTremark}
We will use this theorem in the following form.
Fix any finite sets $\mathcal{M},\mathcal{N}$ such that $\mathcal{N}\subseteq\mathcal{M}\subseteq\Z_{\geq1}$.
Suppose that for all $m\in\mathcal{M}$ there exist monic polynomials $f_m\in\Z[x]$, with splitting fields $K_m$ and Galois groups $G_m=\Gal{\lp K_d/\Q\rp}$.
Suppose that $\prod_{m\in\mathcal{M}}{f_m}$ has no repeated roots and for any $m\in\mathcal{M}$, the field $K_m$ is linearly disjoint from the compositum of $\lb K_{m^\prime}\mid m^\prime\in\mathcal{M}\setminus\lb m\rb\rb$.
Then for any function $\varphi\colon\mathcal{N}\to\Z_{\geq0}$,
\begin{align*}
\delta\lp\bigcap_{i\in\mathcal{N}}\lb p\in\mathcal{P}\mid\lc f_i\rc_p\text{ has precisely }\varphi(i)\text{ roots in }\F_p\rb\rp\hspace{-100px}\\
&=\prod_{i\in\mathcal{N}}{\frac{1}{\lv G_i\rv}\cdot\lv \lb\sigma\in G_i\mid\sigma\text{ fixes exactly $\varphi(i)$ roots of }f_i\rb\rv}.
\end{align*}
\end{FDTremark}

\subsection{How random is the cycle structure of unicritical polynomials?}
\label{cyclestructuremoments}

We now prove our main theorem about the distribution of cycles in $\lc x^k+a\rc_\mathcal{P}$, for $k\in\Z_{\geq2}$ and most $a\in\Z$.
Afterwards, we deduce a density result that encompasses \emph{all} integers $a\in\Z$.

\begin{bigtheorem}\label{bigtheorem}
If $k\in\Z_{\geq2}$ and $\mathcal{N}$ is a finite subset of $\Z_{\geq1}$, then there is a thin set $\mathcal{A}_{k,\mathcal{N}}$ with the property that for any $a\in\Z\setminus\mathcal{A}_{k,\mathcal{N}}$ and any function $\varphi\colon\mathcal{N}\to\Z_{\geq0}$,
\[
\delta\lp\lb p\in\mathcal{P}\mid\text{for all $i\in \mathcal{N}$, $\lp\F_p,\lc x^k+a\rc_p\rp$ has precisely $\varphi(i)$ $i$-cycles}\rb\rp=\prod_{i\in\mathcal{N}}{{}_k\omega_i(\varphi(i))}.
\]
In particular, if $a\in\Z\setminus\mathcal{A}_{k,\mathcal{N}}$ and $i\in\mathcal{N}$, then for all $j\in\Z_{\geq0}$,
\[
\delta\lp\lb p\in\mathcal{P}\mid\lp\F_p,\lc x^k+a\rc_p\rp\text{ has precisely $j$ $i$-cycles}\rb\rp
={}_k\omega_i(j).
\]
\end{bigtheorem}
\begin{proof}
Choose any $N_0>\max{\lp\mathcal{N}\rp}$, let $N=N_0!$, and let let $\mathcal{N}_0$ be the set of positive integer divisors of $N$, so that $\mathcal{N}\subseteq\mathcal{N}_0$.
Set $f=x^k+c\in\Q(c)[x]$, so that if $a\in\Z$, then $f_a=x^k+a\in\Z[x]$.
Next, for any such $a$, let
\[
\mathcal{F}(a)=\lb\Sigma_{f_a,d}
\mid d\in\mathcal{N}_0\rb.
\]
By \ref{mortonpower} and the \hyperref[HIT]{Hilbert Irreducibility Theorem}, there exists a thin set of integers $\mathcal{A}_{k,\mathcal{N}}$ such that for all $a\in\Z\setminus\mathcal{A}_{k,\mathcal{N}}$,
\begin{enumerate}
\item
$\lp f_a\rp^N(x)-x$ has no repeated roots,
\item
any field in $\mathcal{F}(a)$ is linearly disjoint from the compositum of the others,
\item
if $\Sigma_{f_a,d}\in\mathcal{F}\lp a\rp$, then $\Gal{\lp\Sigma_{f_a,d}/\Q\rp}\simeq C_d \wr S_{r_k(d)}$, and
\item
if $\Sigma_{f_a,d}\in\mathcal{F}\lp a\rp$, then the action of $\Gal{\lp\Sigma_{f_a,d}/\Q\rp}$ on the points of period $d$ of $\lp\Sigma_{f_a,d},f\rp$ matches the action of $C_d\wr S_{r_k(d)}$ on $B(d,r_k(d))$.
\end{enumerate}

Fix any $a\in\Z\setminus\mathcal{A}_{k,\mathcal{N}}$.
To ease notation, for any $i\in\mathcal{N}$, we will write $\Phi_i$ for $\Phi_{f_a,i}$.
By \hyperref[rootsandcycles]{Proposition~\ref*{rootsandcycles}} and (1), note that
\begin{align*}
\delta\lp\lb p\in\mathcal{P}\mid\text{for all $i\in\mathcal{N}$, }\lp\F_p,\lc x^k+a\rc_p\rp\text{ has precisely $\varphi(i)$ $i$-cycles}\rb\rp\hspace{-200px}\\
&=\delta\lp\bigcap_{i\in\mathcal{N}}{\lb p\in\mathcal{P}\mid\lp\F_p,\lc x^k+a\rc_p\rp\text{ has precisely $\varphi(i)$ $i$-cycles}\rb}\rp\\
&=\delta\lp\bigcap_{i\in\mathcal{N}}{\lb p\in\mathcal{P}\mid\lc\Phi_i\rc_p\text{ has precisely $i\cdot \varphi(i)$ roots in }\F_p\rb}\rp.
\end{align*}
Since $\mathcal{N}\subseteq\mathcal{N}_0$ by construction, properties (1)--(4) allow us to apply the \ref{FDT} (see \hyperref[FDTremark]{Remark~\ref*{FDTremark}}) to conclude that
\[
\delta\lp\bigcap_{i\in\mathcal{N}}{\lb p\in\mathcal{P}\mid\lc\Phi_i\rc_p\text{ has precisely $i\cdot \varphi(i)$ roots in }\F_p\rb}\rp
=\prod_{i\in\mathcal{N}}{{}_k\omega_i(\varphi(i))}.
\]
\end{proof}
\noindent At this point, \hyperref[finalmomentscor]{Corollary~\ref{finalmomentscor}} is immediate from \hyperref[bigtheorem]{Theorem~\ref*{bigtheorem}} and \hyperref[momentswecareabout]{Corollary~\ref*{momentswecareabout}}.


To analyze the cycle structure of \emph{all} monic binomial unicritical integral polynomials of a fixed degree $k$, we now introduce an asymptotic density on the set of such polynomials.
Let $k\in\Z_{\geq2}$, let $\mathcal{N}$ be a finite subset of $\Z_{\geq1}$, and choose any $\varphi\colon\mathcal{N}\to{\Z_{\geq0}}$.
Define the \emph{truncated cycle density} on $\lb x^k+a\mid a\in\Z\rb$ to be the function ${_k}{\delta}_{\mathcal{N},\varphi}\colon\Z_{\geq0}\to\R_{\geq0}$ given by
\[
{}_k{\delta}_{\mathcal{N},\varphi}(X)
=\frac{1}{2X+1}\sum_{a=-X}^X
{{\delta}\lp\lb p\in\mathcal{P}\mid\text{for all $i\in \mathcal{N}$, $\lp\F_p,\lc x^k+a\rc_p\rp$ has precisely $\varphi(i)$ $i$-cycles}\rb\rp}.
\]
Next, define the \emph{cycle density} on $\lb x^k+a\mid a\in\Z\rb$ to be the function ${}_k\delta_\mathcal{N}\colon{\Z_{\geq0}}^\mathcal{N}\to\R_{\geq0}$ given by
\[
{}_k\delta_\mathcal{N}(\varphi)
=\lim_{X\to\infty}
{}_k{\delta}_{\mathcal{N},\varphi}(X)
\]
(we will show in \cref{cycledistribution} that these limit exists).

\begin{cycledistribution}\label{cycledistribution}
Let $k\in\Z_{\geq2}$, let $\mathcal{N}$ be a finite subset of $\Z_{\geq1}$, and choose any $\varphi\in{\Z_{\geq0}}^\mathcal{N}$.
Then
\[
\lv\prod_{i\in\mathcal{N}}{{}_k\omega_i(\varphi(i))}-{}_k{\delta}_{\mathcal{N},\varphi}(X)\rv
=O_{k,\mathcal{N}}\lp\frac{1}{\sqrt{X}}\rp.
\]
In particular,
\begin{itemize}
\item
the implied constants do not depend on $\varphi$ and
\item
$\displaystyle{
{}_k\delta_\mathcal{N}(\varphi)
=\prod_{i\in\mathcal{N}}{{}_k\omega_i(\varphi(i))}.}$
\end{itemize}
\end{cycledistribution}
\begin{proof}
By \cref{bigtheorem}, there exists a thin set $\mathcal{A}_{k,\mathcal{N}}$ with the property that for any $a\in\Z\setminus\mathcal{A}_{k,\mathcal{N}}$,
\[
\delta\lp\lb p\in\mathcal{P}\mid\text{for all $i\in \mathcal{N}$, $\lp\F_p,\lc x^k+a\rc_p\rp$ has precisely $\varphi(i)$ $i$-cycles}\rb\rp=\prod_{i\in\mathcal{N}}{{}_k\omega_i(\varphi(i))}.
\]
Now, by \cref{HITremark} we know that
\[
\lv\lb a\in\mathcal{A}_{k,\mathcal{N}}\mid -X\leq a\leq X\rb\rv=O\lp\sqrt{X}\rp;
\]
thus,
\begin{align*}
\lv\prod_{i\in\mathcal{N}}{{}_k\omega_i(\varphi(i))}-{}_k{\delta}_{\mathcal{N},\varphi}(X)\rv\hspace{-100px}\\
&\leq
\lv\prod_{i\in\mathcal{N}}{{}_k\omega_i(\varphi(i))}-
\frac{\lv\lb a\in\Z\setminus\mathcal{A}_{k,\mathcal{N}}\mid-X\leq a\leq X\rb\rv\cdot\prod_{i\in\mathcal{N}}{{}_k\omega_i(\varphi(i))}}{2X+1}\rv
+\sum_{\substack{a\in\mathcal{A}_{k,\mathcal{N}}\\-X\leq a\leq X}}{\frac{1}{2X+1}}\\
&=\prod_{i\in\mathcal{N}}{{}_k\omega_i(\varphi(i))}\cdot\frac{2X+1-\lv\lb a\in\Z\setminus\mathcal{A}_{k,\mathcal{N}}\mid-X\leq a\leq X\rb\rv}{2X+1}
+\sum_{\substack{a\in\mathcal{A}_{k,\mathcal{N}}\\-X\leq a\leq X}}{\frac{1}{2X+1}}\\
&=O_{k,\mathcal{N}}\lp\frac{1}{\sqrt{X}}\rp,
\end{align*}
as desired.
\end{proof}

%

\subsection{Most unicritical polynomials have many cycles}
\label{puttingitalltogether}

We now prove \hyperref[recursion]{Lemma~\ref*{recursion}}, which computes asymptotics for a certain family of recursively defined sequences; it will be useful for proving \hyperref[thm:final]{Theorem~\ref*{thm:final}}.

\begin{recursion}\label{recursion}
For all $j\in\Z_{\geq0}$, let $\lp t_{j,n}\rp_{n\in\Z_{\geq1}}$ be a sequence of nonnegative real numbers.
For all such $j$, let $\lp s_{j,n}\rp_{n\in\Z_{\geq1}}$ be the sequence defined recursively by
\[
s_{j,n}=
\begin{cases}
t_{j,1}&\text{if $n=1$,}\\
\sum_{i=0}^j{s_{i,n-1}t_{j-i,n}}&\text{otherwise.}
\end{cases}
\]
If
\[
\text{for all }j\in\Z_{\geq0},
\hspace{15px}
t_{j,n}=\frac{1}{j!n^j}\lp1-\frac{1}{n}\rp+O\lp\frac{1}{n^{j+2}}\rp,
\]
then
\[
\text{for all }j\in\Z_{\geq0},
\hspace{15px}
s_{j,n}=O\lp n^{-\frac{1}{j+1}}\rp.
\]
\end{recursion}
\begin{proof}
We will induct on $j$.
For the 
$j=0$ case, we begin by noting that $s_{0,n}=\prod_{i=1}^n{t_{0,i}}$ for all $n\in\Z_{\geq1}$.
Choose any $R\in\R_{>0}$ such that for all $n\in\Z_{\geq1}$,
\[
t_{0,n}<1-\frac{1}{n}+\frac{R}{n^2}.
\]
Then for all such $n$,
\[
0\leq ns_{0,n}
<n\prod_{i=1}^n{\lp1-\frac{1}{i}+\frac{R}{i^2}\rp}
=R\prod_{i=2}^n{\lp\frac{i-1+\frac{R}{i}}{i-1}\rp}
=R\prod_{i=2}^n{\lp1+\frac{R}{i(i-1)}\rp}.
\]
Since $\prod_{i=2}^\infty {\lp1+\frac{R}{i(i-1)}\rp}$ converges, we see $s_{0,n}=O\lp\frac{1}{n}\rp$, as desired.

For the induction step, suppose that $j\in\Z_{\geq1}$ and
\[
s_{j^\prime,n}=O\lp n^{-\frac{1}{j^\prime+1}}\rp
\hspace{10px}
\text{for all }j^\prime\in\lb0,\ldots,j-1\rb.
\]
By the induction hypothesis, we see that
\[
\sum_{i=0}^{j-1}{s_{i,n-1}t_{j-i,n}}=O\lp n^{-1-\frac{1}{j}}\rp.
\]
Choose any $R\in\R_{>0}$ such that
\begin{itemize}
\item
$\sum_{i=0}^{j-1}{s_{i,n-1}t_{j-i,n}}<Rn^{-1-\frac{1}{j}}$ for all $n\in\Z_{\geq1}$ and
\item
$t_{0,n}<1-\frac{1}{n}+\frac{R}{n^2}$ for all $n\in\Z_{\geq1}$.
\end{itemize}
Note that since
\[
\lp(n-1)n^{\frac{1}{j+1}}+1\rp^{j+1}
=n^{j+2}-(j+1)n^{j+1}+O\lp n^{j+\frac{j}{j+1}}\rp,
\]
there is some $M\in\Z_{\geq1}$ such that
\[
(n-1)n^{\frac{1}{j+1}}+1<n(n-1)^{\frac{1}{j+1}}
\]
for all $n\in\Z_{\geq M}$.
Now choose any $N\in\Z_{\geq1}$ such that $N>\max{\lp\lb R^{\frac{j+1}{j}},M\rb\rp}$, and note that for any $n\in\Z_{\geq N}$,
\begin{align*}
n^{\frac{1}{j+1}}s_{j,n}
&<\frac{R}{n^{1+\frac{1}{j}-\frac{1}{j+1}}}
+s_{j,n-1}\lp\frac{n-1}{n^{1-\frac{1}{j+1}}}+\frac{R}{n^{2-\frac{1}{j+1}}}\rp&&\text{by choice of }R\\
&<\frac{R}{n^{1+\frac{1}{j(j+1)}}}
+(n-1)^{\frac{1}{j+1}}s_{j,n-1}\lp\lp\frac{n-1}{n}\rp^{1-\frac{1}{j+1}}+\frac{1}{n(n-1)^{\frac{1}{j+1}}}\rp&&\text{since }n>R^\frac{j+1}{j}\\
&<\frac{R}{n^{1+\frac{1}{j(j+1)}}}+(n-1)^{\frac{1}{j+1}}s_{j,n-1}&&\text{since }n>M.
\end{align*}
Finally, since $\sum_{n=1}^\infty{n^{-1-\frac{1}{j(j+1)}}}<\infty$, we see that $n^{\frac{1}{j+1}}s_{j,n}$ is bounded, as desired.
\end{proof}

Before continuing, we pause to introduce a bit of notation.
For any dynamical system $(S,f)$ and positive integer $n$, let $\mathbf{C}_n(S,f)=\lv\lb\text{$n$-cycles in $(S,f)$}\rb\rv$.
We now use \hyperref[bigtheorem]{Theorem~\ref*{bigtheorem}}, along with \hyperref[recursion]{Lemma~\ref*{recursion}} and the bounds computed in \hyperref[probbounds]{Section~\ref*{probbounds}} (in particular, \hyperref[fixedpointprobcor]{Corollary~\ref*{fixedpointprobcor}}) to prove \hyperref[thm:final]{Theorem~\ref*{thm:final}}.

\begin{thm}
\label{thm:final}
If $k\in\Z_{\geq2}$ and $J\in\Z_{\geq1}$, then for any $\epsilon\in\R_{>0}$ there exists a thin set $\mathcal{A}_{k,J,\epsilon}\subseteq\Z$ such that for all $a\in\Z\setminus\mathcal{A}_{k,J,\epsilon}$,
\[
\overline{\delta}\lp\lb p\in\mathcal{P}
\mid\lp\F_p,\lc x^k+a\rc_p\rp\text{ has }J\text{ or fewer cycles}\rb\rp
<\epsilon.
\]
\end{thm}

\begin{proof}

For any $j\in\Z_{\geq0}$ and $n\in\Z_{\geq1}$, let
\begin{itemize}
\item
$\mathcal{J}_{j,n}=\lb \varphi\in {\Z_{\geq0}}^{[n]}\mid\sum_{i=1}^n{\varphi(i)}=j\rb$,
\item
$t_{j,n}={}_k\omega_n(j)$, and
\item
$s_{j,n}=\sum_{\varphi\in\mathcal{J}_{j,n}}
{\prod_{i\in[n]}{{}_k\omega_i\lp \varphi(i)\rp}}$.
\end{itemize}

To see how these quantities are interrelated, first note that for all $j\in\Z_{\geq0}$, they imply that $s_{j,1}=t_{j,1}$; moreover, if $n\in\Z_{\geq2}$, then
\[
s_{j,n}
=\sum_{\varphi\in\mathcal{J}_{j,n}}{\prod_{i\in[n]}{{}_k\omega_i(\varphi(i))}}
=\sum_{m=0}^j
{\sum_{\varphi\in\mathcal{J}_{m,n-1}}
{\lp{}_k\omega_n(j-m)\prod_{i\in[n-1]}{{}_k\omega_i(\varphi(i))}\rp}}
=\sum_{m=0}^j{s_{m,n-1}t_{j-m,n}}.
\]
Now, by \hyperref[fixedpointprobcor]{Corollary~\ref*{fixedpointprobcor}}, we know that for $j\in\Z_{\geq0}$,
\[
t_{j,n}=\frac{1}{j!n^j}\lp1-\frac{1}{n}\rp+O\lp\frac{1}{n^{j+2}}\rp.
\]
Thus, by \hyperref[recursion]{Lemma~\ref*{recursion}}, there exists $N$ in $\Z_{\geq1}$ such that $\sum_{j=0}^J{s_{j,N}}<\epsilon$.
Set $\mathcal{N}=\lc N\rc$; then by \hyperref[bigtheorem]{Theorem~\ref*{bigtheorem}}, there is a thin set $\mathcal{A}=\mathcal{A}_{k,J,\epsilon}\subseteq\Z$ such that for all $a\in\Z\setminus\mathcal{A}$ and any function $\varphi\colon\mathcal{N}\to\Z_{\geq0}$,
\[
\delta\lp\lb p\in\mathcal{P}\mid\text{for all $i\in \mathcal{N}$, $\lp\F_p,\lc x^k+a\rc_p\rp$ has precisely $\varphi(i)$ $i$-cycles}\rb\rp=\prod_{i\in\mathcal{N}}{{}_k\omega_i(\varphi(i))}.
\]
For any such $a$, note that
\begin{align*}
\overline{\delta}\lp\lb p\in\mathcal{P}\,\Big\vert\,
\lp\F_p,\lc x^k+a\rc_p\rp\text{ has $J$ or fewer cycles}\rb\rp\hspace{-200px}\\
&=\overline{\delta}\lp\lb p\in\mathcal{P}\,\bigg\vert\,
\sum_{n=1}^\infty{\mathbf{C}_n\lp\F_p,\lc x^k+a\rc_p\rp}\leq J\rb\rp\\
&\leq\overline{\delta}\lp\lb p\in\mathcal{P}\,\bigg\vert\,
\sum_{n=1}^N{\mathbf{C}_n\lp\F_p,\lc x^k+a\rc_p\rp}\leq J\rb\rp\\
&=\sum_{j=0}^J{\delta\lp\lb p\in\mathcal{P}\,\Big\vert\,
\lp\F_p,\lc x^k+a\rc_p\rp\text{ has precisely $j$ cycles of length at most $N$}\rb\rp}\\
&=\sum_{j=0}^J{\delta\lp\bigcup_{\varphi\in\mathcal{J}_{j,N}}{\lb p\in\mathcal{P}\,\Big\vert\,\text{for all $i\in\mathcal{N}$, }
\lp\F_p,\lc x^k+a\rc_p\rp\text{ has precisely $\varphi(i)$ $i$-cycles}\rb}\rp}\\
&=\sum_{j=0}^J{\sum_{\varphi\in\mathcal{J}_{j,N}}{\delta\lp\lb p\in\mathcal{P}\,\Big\vert\,\text{for all $i\in\mathcal{N}$, }
\lp\F_p,\lc x^k+a\rc_p\rp\text{ has precisely $\varphi(i)$ $i$-cycles}\rb\rp}}\\
&=\sum_{j=0}^J{\sum_{\varphi\in\mathcal{J}_{j,N}}
{\prod_{i\in\mathcal{N}}{{}_k\omega_i\lp \varphi(i)\rp}}},
\end{align*}
so we are done by our choice of $N$.
\end{proof}

As in \hyperref[cyclestructuremoments]{Section~\ref*{cyclestructuremoments}}, we can show that \emph{all} monic binomial unicritical polynomials of a fixed degree $k$ have few cycles, on average, once we introduce the appropriate density functions.
For $k\in\Z_{\geq2}$, $n\in\Z_{\geq1}$, and $J,X\in\Z_{\geq0}$, define
\[
{}_k{\delta}_{n,J}(X)
=\frac{1}{2X+1}\sum_{a=-X}^X
{{\delta}\lp\lb p\in\mathcal{P}\,\bigg\vert\,
\sum_{i=1}^n{\mathbf{C}_i\lp\F_p,\lc x^k+a\rc_p\rp}\leq J\rb\rp}.
\]
\begin{oldbox}\label{oldbox}
Note that for any such $k,n,J,X$,
\[
{}_k{\delta}_{n,J}(X)=\sum_{\substack{\varphi\in\lp\Z_{\geq0}\rp^{[n]}\\\varphi(1)+\cdots+\varphi(n)\leq J}}{{}_k\delta_{[n],\varphi}(X)}.
\]
\end{oldbox}
\noindent As in \hyperref[cyclestructuremoments]{Section~\ref*{cyclestructuremoments}}, let
\[
{}_k\delta_n(J)
=\lim_{X\to\infty}
{}_k{\delta}_{n,J}(X)
\]
(this limit exists by \cref{cycledistribution}.)

\begin{disconnectedbox}\label{disconnectedbox}
If $k\in\Z_{\geq2}$ and $J\in\Z_{\geq1}$, then for any $\epsilon\in\R_{>0}$, there exist $C\in\R_{>0}$ and $N\in\Z_{\geq1}$ such that for all $n\in\Z_{\geq N}$ and $X\in\Z_{\geq1}$,
\[
{}_k\delta_{n,J}(X)
<\epsilon+\frac{C}{\sqrt{X}}.
\]
In particular, for all such $n$,
\[
{}_k\delta_n(J)
\leq\epsilon.
\]
\end{disconnectedbox}
\begin{proof}
Adopting the notation of \cref{thm:final}, we see by \cref{recursion} that there is some $N\in\Z_{>0}$ such that $\sum_{j=1}^J{s_{j,n}}<\epsilon$ for all $n\geq N$.
Thus, we apply \cref{cycledistribution} and \cref{oldbox} to note that for all such $n$,
\[
{}_k\delta_{n,J}(X)
=\sum_{j=1}^J{s_{j,n}}+O\lp\frac{1}{\sqrt{X}}\rp.
\]
\end{proof}

\begin{lotsofcycles}\label{lotsofcycles}
Suppose that $k\in\Z_{\geq2}$.
There is an increasing function $\gamma\colon\Z_{\geq1}\to\Z_{\geq1}$ such that
\[
\lim_{J\to\infty}
{{}_k\delta_{\gamma(J),J}\lp\gamma(J)\rp}
=0.
\]
\end{lotsofcycles}
\begin{proof}
Set $\gamma(1)=1$ and for any $J\in\Z_{\geq2}$, apply \cref{disconnectedbox} to choose $\gamma(J)\in\Z_{\geq1}$ such that
\begin{itemize}
\item
$\displaystyle{\gamma(J)>\gamma(J-1)}$ and
\item
$\displaystyle{{}_k\delta_{\gamma(J),J}\lp\gamma(J)\rp<\frac{1}{J}}$.
\end{itemize}
\end{proof}

%

It is a classical fact that the average number of cycles of a random discrete dynamical system grows logarithmically; indeed Kruskal~\cite{Kruskal} proved that
\[
\frac{1}{\lv[X]^{[X]}\rv}
\cdot\sum_{f\in[X]^{[X]}}
{\lv\lb\text{cycles in }(S,f)\rb\rv}
=\frac{1}{2}\log{X}+\lp\frac{\log{2}+E}{2}\rp+o(1),
\]
where $E=.5772\ldots$ is Euler's constant.
This fact leads us to make the following conjecture about the nature of the function $\gamma$ introduced in \hyperref[lotsofcycles]{Corollary~\ref*{lotsofcycles}}.

\begin{whatisgamma}\label{whatisgamma}
Suppose that $k\in\Z_{\geq2}$ and $\gamma\in\lp\Z_{\geq0}\rp^{\Z_{\geq0}}$.
\begin{itemize}
\item
If $J=o\lp\log{\gamma(J)}\rp$, then there exists an increasing function $\alpha\colon\Z_{\geq0}\to\Z_{\geq0}$ such that
\[
\lim_{J\to\infty}
{{}_k\delta_{\gamma(J),J}\lp\alpha(J)\rp}
=0.
\]
\item
If there is some $\epsilon\in\lp-\infty,1\rp$ such that $\log{\gamma(J)}=O\lp J^\epsilon\rp$, then there exists an increasing function $\alpha\colon\Z_{\geq0}\to\Z_{\geq0}$ such that
\[
\lim_{J\to\infty}
{{}_k\delta_{\gamma(J),J}\lp\alpha(J)\rp}
=1.
\]
\end{itemize}

\end{whatisgamma}

\section*{Acknowledgements} \label{Acknowledgements}

The authors would like to thank Rafe Jones for many useful discussions and for directing us to pertinent background literature.
We also thank Roger Heath-Brown for posing a version of the problem studied in this paper, as well as for helpful conversations on this topic.
Finally, we thank Igor Shparlinski for his comments on an early draft of this paper.

\bibliography{../../CycleStructure}

\newcommand{\etalchar}[1]{$^{#1}$}
\providecommand{\bysame}{\leavevmode\hbox to3em{\hrulefill}\thinspace}
\providecommand{\MR}{\relax\ifhmode\unskip\space\fi MR }
\providecommand{\MRhref}[2]{%
  \href{http://www.ams.org/mathscinet-getitem?mr=#1}{#2}
}
\providecommand{\href}[2]{#2}
\begin{thebibliography}{BGTW18}

\bibitem[AB82]{ArneyBender}
James Arney and Edward~A. Bender,
  \emph{\href{http://projecteuclid.org/euclid.pjm/1102723961}{Random mappings
  with constraints on coalescence and number of origins}}, Pacific J. Math.
  \textbf{103} (1982), no.~2, 269--294. \MR{705228}

\bibitem[Bac91]{Bach}
Eric Bach, \emph{\href{http://dx.doi.org/10.1016/0890-5401(91)90001-I}{Toward a
  theory of {P}ollard's rho method}}, Inform. and Comput. \textbf{90} (1991),
  no.~2, 139--155. \MR{1094034}

\bibitem[BG17]{BG}
Andrew Bridy and Derek Garton,
  \emph{\href{http://dx.doi.org/10.1186/s40687-017-0103-3}{Dynamically
  distinguishing polynomials}}, Res. Math. Sci. \textbf{4} (2017), no.~4,
  1--17. \MR{3669394}

\bibitem[BGH{\etalchar{+}}13]{BGHKST}
Robert~L. Benedetto, Dragos Ghioca, Benjamin Hutz, P\"ar Kurlberg, Thomas
  Scanlon, and Thomas~J. Tucker,
  \emph{\href{http://dx.doi.org/10.1007/s00208-012-0799-8}{Periods of rational
  maps modulo primes}}, Math. Ann. \textbf{355} (2013), no.~2, 637--660.
  \MR{3010142}

\bibitem[BGTW18]{BGTW}
Elisa Bellah, Derek Garton, Erin Tannenbaum, and Noah Walton,
  \emph{\href{http://dx.doi.org/10.2140/involve.2018.11.169}{A probabilistic
  heuristic for counting components of functional graphs of polynomials over
  finite fields}}, Involve \textbf{11} (2018), no.~1, 169--179. \MR{3681355}

\bibitem[BJ07]{BostonJones}
Nigel Boston and Rafe Jones,
  \emph{\href{http://dx.doi.org/10.1007/s10711-006-9113-9}{Arboreal {G}alois
  representations}}, Geom. Dedicata \textbf{124} (2007), 27--35.

\bibitem[BP81]{BP}
Richard~P. Brent and John~M. Pollard,
  \emph{\href{http://dx.doi.org/10.2307/2007666}{Factorization of the eighth
  {F}ermat number}}, Math. Comp. \textbf{36} (1981), no.~154, 627--630.
  \MR{606520}

\bibitem[BS17]{BS}
Charles Burnette and Eric Schmutz,
  \emph{\href{http://dx.doi.org/10.1142/S1793042117500713}{Periods of iterated
  rational functions}}, Int. J. Number Theory \textbf{13} (2017), no.~5,
  1301--1315. \MR{3639698}

\bibitem[CM12]{ChowMansour}
Chak-On Chow and Toufik Mansour,
  \emph{\href{http://web.math.rochester.edu/misc/ojac/vol7/Mansour_2012.pdf}{Asymptotic
  probability distributions of some permutation statistics for the wreath
  product {$C_r\wr S_n$}}}, Online J. Anal. Comb. (2012), no.~7, 14.
  \MR{3016122}

\bibitem[Coh81]{CohenHIT}
S.~D. Cohen, \emph{\href{http://dx.doi.org/10.1112/plms/s3-43.2.227}{The
  distribution of {G}alois groups and {H}ilbert's irreducibility theorem}},
  Proc. London Math. Soc. (3) \textbf{43} (1981), no.~2, 227--250. \MR{628276}

\bibitem[FG14]{FG}
Ryan Flynn and Derek Garton,
  \emph{\href{http://dx.doi.org/10.1142/S1793042113501224}{Graph components and
  dynamics over finite fields}}, Int. J. Number Theory \textbf{10} (2014),
  no.~3, 779--792. \MR{3190008}

\bibitem[FO90]{FlajoletOdlyzko}
Philippe Flajolet and Andrew~M. Odlyzko,
  \emph{\href{http://dx.doi.org/10.1007/3-540-46885-4_34}{Random mapping
  statistics}}, Advances in cryptology---{EUROCRYPT} '89 ({H}outhalen, 1989),
  Lecture Notes in Comput. Sci., vol. 434, Springer, Berlin, 1990,
  pp.~329--354. \MR{1083961}

\bibitem[Gon44]{Gon}
V.~Gontcharoff, \emph{\href{http://mi.mathnet.ru/eng/izv3716}{Du domaine de
  l'analyse combinatoire}}, Bull. Acad. Sci. URSS S\'er. Math. [Izvestia Akad.
  Nauk SSSR] \textbf{8} (1944), 3--48. \MR{0010922}

\bibitem[Gon62]{GonTran}
V.~Gon\v{c}arov, \emph{\href{http://bookstore.ams.org/trans2-19/}{On the field
  of combinatory analysis}}, Amer. Math. Soc. Transl. (2) \textbf{19} (1962),
  1--46. \MR{0131369}

\bibitem[Har60]{Harris}
Bernard Harris,
  \emph{\href{https://doi.org/10.1214/aoms/1177705677}{Probability
  distributions related to random mappings}}, Ann. Math. Statist. \textbf{31}
  (1960), 1045--1062. \MR{0119227}

\bibitem[HB17]{Heath-Brown}
D.~R. Heath-Brown,
  \emph{\href{https://doi.org/10.1112/S0025579317000328}{Iteration of Quadratic
  Polynomials Over Finite Fields}}, Mathematika \textbf{63} (2017), no.~3,
  1041--1059. \MR{3731313}

\bibitem[Isa06]{IsaacsCharacterTheory}
I.~Martin Isaacs, \emph{\href{http://dx.doi.org/10.1090/chel/359}{Character
  theory of finite groups}}, AMS Chelsea Publishing, Providence, RI, 2006,
  Corrected reprint of the 1976 original [Academic Press, New York; MR0460423].
  \MR{2270898}

\bibitem[Isa08]{IsaacsGroupTheory}
\bysame, \emph{\href{http://dx.doi.org/10.1090/gsm/092}{Finite group theory}},
  Graduate Studies in Mathematics, vol.~92, American Mathematical Society,
  Providence, RI, 2008. \MR{2426855}

\bibitem[JKMT16]{JKMT}
Jamie Juul, P\"ar Kurlberg, Kalyani Madhu, and Tom~J. Tucker,
  \emph{\href{http://dx.doi.org/10.1093/imrn/rnv273}{Wreath products and
  proportions of periodic points}}, Int. Math. Res. Not. IMRN (2016), no.~13,
  3944--3969. \MR{3544625}

\bibitem[Kru54]{Kruskal}
Martin~D. Kruskal, \emph{\href{https://doi.org/10.2307/2307900}{The expected
  number of components under a random mapping function}}, Amer. Math. Monthly
  \textbf{61} (1954), 392--397. \MR{0062973 (16,52b)}

\bibitem[Lan83]{LangDiophantineGeometry}
Serge Lang,
  \emph{\href{http://dx.doi.org/10.1007/978-1-4757-1810-2}{Fundamentals of
  {D}iophantine geometry}}, Springer-Verlag, New York, 1983. \MR{715605}

\bibitem[Mor96]{MorCurves}
Patrick Morton,
  \emph{\href{http://www.numdam.org/item?id=CM_1996__103_3_319_0}{On certain
  algebraic curves related to polynomial maps}}, Compositio Math. \textbf{103}
  (1996), no.~3, 319--350. \MR{1414593 (97m:14030)}

\bibitem[Mor98]{MorGalGroups}
\bysame, \emph{\href{http://dx.doi.org/10.1006/jabr.1997.7304}{Galois groups of
  periodic points}}, J. Algebra \textbf{201} (1998), no.~2, 401--428.
  \MR{1612390}

\bibitem[Mor11]{MorCorrection}
\bysame, \emph{\href{http://doi.org/10.1112/S0010437X1000480X}{Corrigendum:
  `{O}n certain algebraic curves related to polynomial maps, {C}ompositio
  {M}ath. 103 (1996), 319--350'}}, Compos. Math. \textbf{147} (2011), no.~1,
  332--334. \MR{2771135}

\bibitem[MP94]{MP}
Patrick Morton and Pratiksha Patel,
  \emph{\href{http://dx.doi.org/10.1112/plms/s3-68.2.225}{The {G}alois theory
  of periodic points of polynomial maps}}, Proc. London Math. Soc. (3)
  \textbf{68} (1994), no.~2, 225--263. \MR{1253503}

\bibitem[MSSS]{MSSS}
B.~{Mans}, M.~{Sha}, I.~E. {Shparlinski}, and D.~{Sutantyo},
  \emph{\href{https://doi.org/10.1080/10586458.2017.1391725}{On Functional
  Graphs of Quadratic Polynomials}}, to appear in Exp. Math.

\bibitem[Odo85]{Odoni}
R.~W.~K. Odoni, \emph{\href{http://dx.doi.org/10.1112/plms/s3-51.3.385}{The
  {G}alois theory of iterates and composites of polynomials}}, Proc. London
  Math. Soc. (3) \textbf{51} (1985), no.~3, 385--414.

\bibitem[Pol75]{Pollard}
J.~M. Pollard, \emph{\href{http://dx.doi.org/10.1007/BF01933667}{A {M}onte
  {C}arlo method for factorization}}, Nordisk Tidskr. Informationsbehandling
  (BIT) \textbf{15} (1975), no.~3, 331--334. \MR{0392798}

\bibitem[Ser97]{Serre}
Jean-Pierre Serre,
  \emph{\href{http://dx.doi.org/10.1007/978-3-663-10632-6}{Lectures on the
  {M}ordell-{W}eil theorem}}, third ed., Aspects of Mathematics, Friedr. Vieweg
  \& Sohn, Braunschweig, 1997, Translated from the French and edited by Martin
  Brown from notes by Michel Waldschmidt, With a foreword by Brown and Serre.
  \MR{1757192}

\bibitem[Sil07]{SilvermanADS}
Joseph~H. Silverman,
  \emph{\href{http://dx.doi.org/10.1007/978-0-387-69904-2}{The arithmetic of
  dynamical systems}}, Graduate Texts in Mathematics, vol. 241, Springer, New
  York, 2007. \MR{2316407}

\bibitem[Sil08]{SilvermanVariationmodp}
\bysame, \emph{\href{http://nyjm.albany.edu/j/2008/14_601.html}{Variation of
  periods modulo {$p$} in arithmetic dynamics}}, New York J. Math. \textbf{14}
  (2008), 601--616. \MR{2448661}

\bibitem[SL96]{SL}
P.~Stevenhagen and H.~W. Lenstra, Jr.,
  \emph{\href{http://websites.math.leidenuniv.nl/algebra/chebotarev.pdf}{Chebotar\"ev
  and his density theorem}}, Math. Intelligencer \textbf{18} (1996), no.~2,
  26--37. \MR{1395088}

\bibitem[Ste69]{Stepanov}
V.~E. Stepanov, \emph{\href{http://doi.org/10.1137/1114078}{Limit distributions
  of certain characteristics of random mappings}}, Teor. Verojatnost. i
  Primenen. \textbf{14} (1969), 639--653. \MR{0278350}

\bibitem[V{\"o}l96]{Volklein}
Helmut V{\"o}lklein,
  \emph{\href{http://dx.doi.org/10.1017/CBO9780511471117}{Groups as {G}alois
  groups}}, Cambridge Studies in Advanced Mathematics, vol.~53, Cambridge
  University Press, Cambridge, 1996, An introduction. \MR{1405612}

\end{thebibliography}
\bibliographystyle{amsalpha}

\end{document}